\documentclass[12pt,reqno]{amsart}
  
  \usepackage{latexsym} 
  \usepackage[all]{xy}
  \usepackage{mathtools}
  \usepackage{amsfonts} 
  \usepackage{amsthm} 
  \usepackage{amsmath} 
  \usepackage{amssymb}
  \usepackage{pifont}  
  \usepackage{enumerate}  
   \usepackage{dcolumn}
  \usepackage{comment}
  \usepackage{lmodern}
\usepackage{microtype}
\usepackage{latexsym} 
\usepackage[dvipsnames]{xcolor}
\newcolumntype{2}{D{.}{}{2.0}}
  \xyoption{2cell}

 \usepackage{tikz}
\usetikzlibrary{arrows}

  \def\<{{\langle}} 
  \def\>{{\rangle}}

  \def\note#1{{}}

  \def\note#1{} 

  \def\beq{\begin{equation}} 
  \def\eeq{\end{equation}}

  \newcounter{zlist}

  \newcounter{blist}

  \newcounter{rlist}

\def\stac#1{\raise-.2cm\hbox{$\stackrel{\displaystyle\otimes}{\scriptscriptstyle{#1}}$}}

\def\cten#1{\raise-.2cm\hbox{$\stackrel{\displaystyle\widehat{\otimes}}
{\scriptscriptstyle{#1}}$}}

  \def\Label#1{\label{#1}\ifmmode\llap{[#1] }\else 
  \marginpar{\smash{\hbox{\tiny [#1]}}}\fi} 
  \def\Label{\label}

  \newtheorem{proposition}{Proposition}[section]
  \newtheorem{lemma}[proposition]{Lemma} 
  \newtheorem{corollary}[proposition]{Corollary} 
  \newtheorem{theorem}[proposition]{Theorem} 
  
\theoremstyle{definition}

  \theoremstyle{remark}

   \numberwithin{equation}{section}

\newcommand{\gG}{\mathrm{G}}
\newcommand{\hH}{\mathrm{H}}

\newcommand{\uU}{\mathrm{U}}

\newcommand{\Cc}{\mathcal{C}}

\newcommand{\Gg}{\mathcal{G}}
\newcommand{\Hh}{\mathcal{H}}

\def\*C{{}^*\hspace*{-1pt}{\Cc}}

\def\text#1{{\rm {\rm #1}}}

\def\Set{\mathbf{Set}}

 \def\1{\mathbf{1}}

\def\hrd {\mathbf{Heap}}
\def\ahrd{\mathbf{Ah}}
\def\grp {\mathbf{Grp}}

\def\lto{\longmapsto}
\def\lra{\longrightarrow}

\def\1\mathbf{1}

\def\|#1{\overline{#1}}

\usepackage{scalerel,stackengine}
\stackMath
\newcommand\reallywidehat[1]{%
\savestack{\tmpbox}{\stretchto{%
  \scaleto{%
    \scalerel*[\widthof{\ensuremath{#1}}]{\kern.1pt\mathchar"0362\kern.1pt}%
    {\rule{0ex}{\textheight}}
  }{\textheight}%
}{2.4ex}}%
\stackon[-6.9pt]{#1}{\tmpbox}%
}
\begin{document}

\title[A note on a free group]{A note on a free group. The decomposition of a free group functor through the category of heaps}

\author{Bernard  Rybo{\l}owicz}

\address{
Department of Mathematics, Swansea University, 
Swansea University Bay Campus,
Fabian Way,
Swansea,
  Swansea SA1 8EN, U.K.}

\urladdr{https://sites.google.com/view/bernardrybolowicz/}
\email{Bernard.Rybolowicz@swansea.ac.uk}



\begin{abstract}
This note aims to introduce a left adjoint functor to the functor which assigns a heap to a group. The adjunction is monadic. It is explained how one can decompose a free group functor through the previously introduced adjoint and employ it to describe a slightly different construction of free groups.
\end{abstract}    
\date\today
\maketitle

\section{Introduction}

In the construction of a free group (see \cite{von} and \cite[Chapter 6]{DF}), one extends a generating set by a neutral element and inverse elements. The procedure of adding those elements happen entirely by using the set-theoretic operations, i.e. a disjoint union of sets. Therefore it is hard to grasp the meaning of those operations in an algebraic sense.
Thus the natural question is if there exists some approach/point of view from which one could spot the algebraic interpretation of the set-theoretic process of extending the generating set.

 The approach presented in the paper uses heaps, a specific variant of universal algebras introduced by R.\ Baer  \cite{Bae:ein} and H.\ Pr\"ufer \cite{Pru:the}. Due to both Pr\"ufer \cite{Pru:the} (Abelian case) and R.\ Baer  \cite{Bae:ein} (general case), it is known that with every heap one can associate a group and that to every group one can assign a heap. The latter is a functor.

This paper aims to construct a left adjoint to the functor, which assigns a heap to a group, see Theorem~\ref{thm1}. The second goal is to point out what consequences it yields for free groups, see Section~\ref{sect4}. The crucial observations of the main part are Corollary~\ref{cor3} and Corollary~\ref{cor4}, which explains that the free group functor decomposes through the category of heaps.

In the conclusions, we explain that even though in general, the description of a coproduct, which is crucial to define the left adjoint, is not an easy task, in this specific case the structure is transported by the free functor of heaps and therefore is easily described. The last words of the paper briefly explain how the inverses and a neutral element arise through the employment of the free group functor.

 \section{Preliminaries}
Following Baer \cite{Bae:ein} and Pr\"ufer \cite{Pru:the} a {\em heap} is a set $H$ together with a ternary operation $[-,-,-]:H\times H \times H\to H$ such that for all $h_1,h_2,h_3,h_4,h_5\in H$ the following holds
\begin{equation}\label{rule:1}
[[h_1,h_2,h_3],h_4,h_5]=[h_1,h_2,[h_3,h_4,h_5]],
\end{equation}
\begin{equation}\label{malcev}
[h_1,h_2,h_2]=h_1=[h_2,h_2,h_1].
\end{equation}
First equation is called an associativity and the second is Mal'cev identities.

If for all $h_1,h_2,h_3\in H$, $[h_1,h_2,h_3]=[h_3,h_2,h_1]$ we say that $H$ is Abelian.

A {\em sub-heap} $S$ of a heap $H$ is a subset of $H$ closed under the ternary operation.

A {\em homomorphism of heaps }is a map between heaps which preserves ternary operation. Observe that a constant map between two heaps is a heap homomorphism since, by Mal'cev identities \eqref{malcev}, a single element $e\in H$ is a sub-heap of $H$.

A sub-heap $S$ is said to be {\em normal} if there exists $e\in S$ such that for all $h \in H$ and $s \in S$ there exists $s' \in S$ such that
$$
[h, e, s] = [s', e, h] \textrm{ or equivalently } [[h,e,s],h,e]=s'
$$
If $S$ is a normal sub-heap then the quotient heap $H/S$ is well defined and canonical map $\pi:H\to H/S$ is a heap epimorphism, see \cite[Proposition 2.10]{Brz:par}.

An important property of heaps is that with every heap $H$ we can associate a group by choosing an element $e\in H$ and defining a binary operation $+_e:=[-,e,-]:H\times H\to H$, we will call the group $(H,+_e)$ {\em a retract of $H$ in $e$} or $e$-{\it retract} and denote by $\gG(H;e)$. It is worth to mention that an assignment $\gG$ is not a functor, as it is not well-defined on the morphisms. If $\varphi$ is a homomorphism of heaps then $\varphi$ is a homomorphism of appropriate retracts if and only if it preserves neutral elements of the retracts.

In the opposite direction, one can associate with every group $G$ a heap by defining ternary operation, for all $a,b,c\in G$, as $[a,b,c]:=ab^{-1}c$. We call this heap a {\em heap associated with a group G} and denote it by $\hH(G)$. In contrast to the assignment $\gG$, the assignment $\hH:\grp\to \hrd$, between categories of groups and heaps, is a functor given on morphisms $\varphi:G\to G'$ by $\hH(\varphi)=\varphi$. Every group homomorphism is a homomorphism of associated heaps. 

By employing both assignments to a group $G$ one gets that a group  $\gG(\hH(G);e)$ is isomorphic to $G$, for all $e\in G$. By applying assignments to a heap $H$, we get that $\hH(\gG(H;e))=H$, for all $e\in H$. The second link between groups and heaps can be used to show that for all $h_1,h_2,h_3,h_4,h_5\in H,$
\begin{equation}\label{rule:2}
[[h_1,h_2,h_3],h_4,h_5]=[h_1,[h_4,h_3,h_2],h_5]=[h_1,h_2,[h_3,h_4,h_5]],
\end{equation}
see Lemma~2.3 of \cite{Brz:par}.

In Section 3 of \cite{BrzRyb:mod}, one can find the construction of a free heap over a set $X$, $\Hh(X)$. Since heaps form a variety of algebras a free functor exists, we will denote it by $\Hh:\Set\to\hrd$ to be coherent with the notation of free heaps. It is a left adjoint to the forgetful functor $\uU_{\hrd}:\hrd \to\Set$.  In contrast to the case of a free group, in a free heap  one does not extend the set of letters by any new letters such as a neutral element or inverse letters.
	
	Since heaps form a variety of algebras small colimits in $\hrd$ exist (see \cite[Theorem 9.4.14]{Ber:inv}), particularly the one in the centre of our attention will be a coproduct.

A {\em coproduct} of two objects $A$ and $B$ in a category $\mathfrak{C}$ is an object $C$ with two morphisms $\iota_A:A\to C$ and $\iota_B:B\to C$, called canonical injections, such that for any object $D$ and morphisms $f:A\to D$ and $g:B\to D$, there exists a unique morphism $\varphi:C\to D$ such that $\varphi\circ \iota_A=f$ and $\varphi\circ \iota_B=g$. In the diagram-like style, the diagram

\begin{equation}\label{sum.diag}
\xymatrix{&& D && \cr A \ar[rr]^{\iota_A}\ar[urr]^f & &C \ar@{-->}[u]_{\varphi} & & B\ar[ll]_{\iota_B}\ar[ull]_g }
\end{equation}
commutes.

In Section 3 of \cite{BrzRyb:mod} one can find the construction of a coproduct in the {\em category of Abelian heaps} $\ahrd$, a full subcategory of $\hrd$. The idea is to take two Abelian heaps $A$ and $B$, consider the free heap over theirs disjoint union $\Hh(A\sqcup B)$ and then divide it by the sub-heap generated by
\begin{equation}\label{eq:1}
[[a,a',a''],[a,a',a'']_{A},e],\quad [[b,b',b''],[b,b',b'']_{B},e],
\end{equation}
for all $a,a',a''\in A$, $b,b',b''\in B$, where
$[---],[---]_{A},[---]_{B}$ are ternary operations in $\Hh(A\sqcup B)$, $A$ and $B$, respectively.

\section{Main part}
Let us fix some notation. Following \cite{BrzRyb:mod} we will denote the coproduct of two not necessarily Abelian heaps $H$ and $S$ by $H\boxplus S$. It exists by \cite[Theorem 9.4.14]{Ber:inv}, though we do not know exact construction of it for non-Abelian heaps, for Abelian case see \cite[Section 3]{BrzRyb:mod}. The unique filler of the coproduct diagram for morphisms $f$ and $g$, will be called a {\em coproduct map} and will be denoted by $f\boxplus g$.

Our first and main goal is to construct a left adjoint functor to the functor $\hH:\grp\to \hrd$.

A singleton heap is a heap that has only one element, we will denote it by $\{*\}$. For any heap $H$, one can consider a group $\mathrm{Gr}_{*}(H):=\gG(H\boxplus\{*\};*)$. The following lemma shows that this group has a very interesting universal property, which will be essential in the construction of the adjoint.

\begin{lemma}\label{heap:lem:univ}
Let $H$ be a heap, S be a group and $f: H\to \hH(S)$ be a heap homomorphism. Then there exists a unique group homomorphism $\mathrm{Gr}_{*}(f):\mathrm{Gr}_{*}(H)\to S$ such that $f=\hH(\mathrm{Gr}_{*}(f))\circ \iota_{H}$, where $\iota_H$ is a canonical injection into coproduct. In other words, diagram 
\begin{equation}\label{graf1}
\xymatrix {H\ar[rr]^-{\iota_{H}} \ar[rrdd]_-f&& \hH(\mathrm{Gr}_{*}(H))  \ar@{-->}[dd]^-{\exists !\, \hH(\mathrm{Gr}_{*}(f))}
\\
\\
&& \hH(S) &}
\end{equation}
commutes, where $\exists !\, \hH(\mathrm{Gr}_{*}(f))$ reads ``\it{there exists exactly one \textbf{homomorphism of groups} $\mathrm{Gr}_{*}(f)$}". The pair $(\mathrm{Gr}_{*}(H),\iota_H)$ is a universal arrow, see \cite[Section III.1]{Mac:lane}.
\end{lemma}
\begin{proof}
Observe that by the universal property of coproduct for all groups $S$ and homomorphisms of heaps $f:H\to \hH(S)$, $g:\{*\}\to \hH(S)$ diagram
\begin{equation}
\xymatrix{&& \hH(S)&& \\ H \ar[rr]^{\iota_H}\ar[urr]^{f} & &\hH(\mathrm{Gr}_{*}(H)) \ar@{-->}[u]_{\hH(\mathrm{Gr}_{*}(f))} & & \{*\}\ar[ll]_{\iota_*}\ar[ull]_{g} }
\end{equation}
commutes.  Every homomorphism of groups is a homomorphism of associated heaps. Moreover, a homomorphism of heaps is a homomorphism of retracts if, and only if it maps a neutral element to a neutral element. Hence, $\hH(\mathrm{Gr}_{*}(f))$ is a homomorphism of retracts if and only if $g(\iota_*(*))$ is a neutral element of $S$. Observe that $g$ is unique, since $\{*\}$ is a singleton heap. Therefore $\hH(\mathrm{Gr}_{*}(f))$ is a unique homomorphism of heaps such that it is also a homomorphism of groups to which heaps were associated. Thus, the preceding diagram commutes.
\end{proof}

Another important observation is that a canonical injection $\iota_H$ has some sort of cancellation property.

\begin{lemma}\label{heap:lem:stepi}
Let $H,L$ be heaps and $f,g:\hH(\mathrm{Gr}_{*}(H))\to L$ be homomorphisms of heaps such that $f(\iota_*(*))=g(\iota_*(*))$, then $f\circ \iota_H=g\circ \iota_H$ implies  $f=g$.
\end{lemma}
\begin{proof}
Let us consider a homomorphism of heaps $f:\hH(\mathrm{Gr}_{*}(H))\to L$. One can easily observe that by the uniqueness of a coproduct map $f=(f\circ\iota_H)\boxplus (f\circ\iota_{*})$. Thus, because $f(\iota_*(*))=g(\iota_{*}(*))$ and $f\circ \iota_H=g\circ \iota_H$, we get that
$$
f=(f\circ\iota_H)\boxplus (f\circ\iota_{*})=(g\circ\iota_H)\boxplus (g\circ\iota_{*})=g.
$$
Therefore, $f(\iota_*(*))=g(\iota_*(*))$ and $(f\circ \iota_H)=(g\circ \iota_H)$ implies  $f=g$.
\end{proof}
\begin{corollary}\label{heap:cor:stepi}
Let $e\in L$. If $f,g:\mathrm{Gr}_{*}(H)\to \gG(L,+_e)$, are homomorphisms of groups, then $f\circ\iota_H=g\circ\iota_H$ implies $f=g$.
\end{corollary}
\begin{proof}
This follows by Lemma \ref{heap:lem:stepi} since a homomorphism of heaps $\hH(f)$ is equal to a homomorphism of groups $f$ as functions.
\end{proof}

The following lemma follows by Theorem 2 (ii) in \cite[Section IV]{Mac:lane}, but for the sake of the unaccustomed reader, we sketch a proof.

Now, we are ready to describe the functor. Let us consider an assignment \linebreak$\mathrm{Gr}:\hrd\to \grp$ given on a heap $H$ by $H\mapsto \mathrm{Gr}_{*}(H)$. One can easily see that it is a well-defined function. The assignment is given for all homomorphisms of heaps $f:H\to H'$ by $f\mapsto \mathrm{Gr}_{*}(\iota_{H'}\circ f)$. The assignment on morphisms is well-defined since $\iota_H'\circ f$ is a composition of homomorphisms of heaps, so it is a homomorphism of heaps. Therefore by the universal property of $\mathrm{Gr}_{*}$, $\mathrm{Gr}_{*}(\iota_{H'}\circ f)$ is a homomorphism of groups.

\begin{lemma}
The assignment $\mathrm{Gr}:\hrd\to \grp$ is a functor.
\end{lemma}

\begin{proof}
In the previous discussion, we explained that both assignments are well-defined functions. Thus, we have to show that functor preserves identity and composition.

Obviously $\mathrm{Gr}_{*}(\iota_{H}\circ id_H)=id_{\mathrm{Gr}_{*}(H)}$.

For the composition let us assume that $f:H\to H'$ is a homomorphism of heaps, then $\iota_{H'}\circ f$ is a composition of homomorphisms of heaps, hence $\iota_{H'}\circ f:H\to\hH(\mathrm{Gr}(H')))$ is a homomorphism of heaps. If $f:H\to H'$ and $g:H'\to H''$ are homomorphisms of heaps, then 
$$
\begin{aligned}
\mathrm{Gr}(g\circ f)\circ \iota_{H}&=\mathrm{Gr}_{*}(\iota_{H''}\circ g\circ f)\circ \iota_{H}=\iota_{H''}\circ g\circ f=\mathrm{Gr}(g)\circ \iota_{H'}\circ f\\ &=\mathrm{Gr}(g)\circ \mathrm{Gr}(f)\circ \iota_{H},
\end{aligned}
$$
where all the equalities follow by Lemma \ref{heap:lem:univ} applied multiple times. Now, since $\mathrm{Gr}(g\circ f)\circ \iota_{H}=\mathrm{Gr}(g)\circ \mathrm{Gr}(f)\circ \iota_{H}$ and $\mathrm{Gr}(g\circ f),\mathrm{Gr}(g), \mathrm{Gr}(f)$ are homomorphisms of groups, applying Corollary \ref{heap:cor:stepi}, one gets that  $\mathrm{Gr}(g\circ f)=\mathrm{Gr}(g)\circ \mathrm{Gr}(f)$. Therefore an assignment $\mathrm{Gr}$ preserves composition, hence $\mathrm{Gr}$ is a functor.
\end{proof}

The following theorem confirms that $\mathrm{Gr}$ is a desirable functor and follows by the Theorem 2 (i) in \cite[Section IV]{Mac:lane}, but for the sake of the unaccustomed reader, we sketch a proof.

\begin{theorem}\label{thm1}
A functor $\mathrm{Gr}$ is a left adjoint to the functor $\hH$.
\end{theorem}
\begin{proof}
 For all heaps $H$ and groups $G$ let us consider functions between sets of morphisms:
 $$
 \varphi_{H,G}:\grp(\mathrm{Gr}(H),G)\lra \hrd(H,\hH(G)),\ \  
 f\lto \hH(f)\circ \iota_{H},
 $$ 
 $$\varphi_{H,G}^{-1}:\hrd(H,\hH(G))\lra \grp(\mathrm{Gr}(H),G),\ \  f\mapsto \mathrm{Gr}_{*}(f).$$ 
 
 To show that $\varphi_{H,G}$ is a bijection let $f\in \hrd(H,\hH(G))$ and $g\in \grp(\mathrm{Gr}(H),G)$, then
 $$
 \varphi_{H,G}\circ \varphi_{H,G}^{-1}(f)=\varphi_{H,G}(\mathrm{Gr}_{*}(f))=\hH(\mathrm{Gr}_{*}(f))\circ\iota_H=f,
 $$
 where the last equality follows by Lemma~\ref{heap:lem:univ}, and
 $$
 \varphi_{H,G}^{-1}\circ \varphi_{H,G}(g)=\mathrm{Gr}_{*}(\hH(g)\circ\iota_H)=g,
 $$
 where the last equality follows by the uniqueness of the morphism $\mathrm{Gr}_{*}(\hH(f)\circ\iota_H)$. 
Hence, $\varphi_{H,G}^{-1}$ is an inverse to $\varphi_{H,G}$. Thus, $\varphi_{H,G}$ is a bijection.

To check naturality conditions, let $G,S$ be groups , $H,L$ be heaps and consider homomorphisms $f:\mathrm{Gr}(H)\lra G$, $\alpha:L\lra H$ and $g:G\lra S$. Then
$$
\varphi_{L,G}(f\circ \mathrm{Gr}(\alpha))=\hH(f\circ \mathrm{Gr}(\alpha))\circ\iota_{L}=\hH(f)\circ \hH(\mathrm{Gr}(\alpha))\circ\iota_{L}=\hH(f)\circ\iota_{H}\circ \alpha=\varphi_{H,G}(f)\circ \alpha,
$$
by applying Lemma~\ref{heap:lem:univ} multiple times.
Similarly,
$$
\varphi_{H,S}(g\circ f)=\hH(g\circ f)\circ \iota_{H}=\hH(g)\circ \varphi_{H,G}(f).
$$
Therefore $\varphi$ is a natural isomorphism and the functor $\mathrm{Gr}$ is a left adjoint to the functor{ }~$\hH$.
 \end{proof}
\begin{proposition}\label{cor:monad}
The adjunction $\mathrm{Gr} \dashv \hH$ is monadic.
\end{proposition}
\begin{proof}
By \cite[3.14. Theorem]{TTT} it is enough to show that $\hH$ reflects isomorphisms, $\grp$ has coequalizers of $\hH$-split parallel pairs, and $\hH$ preserves those coequalizers. 

Let us start with the property of reflecting isomorphisms. Let $g:G\to G'$ be a homomorphism of groups such that $\hH(g)$ is an isomorphism of heaps. Then, since $\hH(g)=g$ as functions there exists a function which is an inverse to $g$, i.e. $\hH(g)^{-1}$. It is obvious that inverse is a group homomorphism. Therefore, $\hH$ reflects isomorphisms.

To prove that coequalizers for $\hH$-split parallel pairs exist let $f,g:G\to G'$ be an $\hH$-split parallel pair, then by definition there exist heap $H$ and heap homomorphisms $$h:\hH(G')\to H,\ t:H\to\hH(G'),\ s:\hH(G')\to \hH(G)$$ such that $h\circ \hH(f)=h\circ\hH(g)$, $s$ and $t$ are sections of $\hH(f)$ and $h$, respectively, and $\hH(g)\circ s=t \circ h$. 
One can check that existence of those homomorphisms imply that a pair $(H,h)$ is a coequalizer of $\hH(f)$ and $\hH(g)$. Now, let us denote by $e\in G$ and $e'\in G'$ neutral elements of the groups, then we can consider an $h(e')$-retract of $H$, a group $\gG(H;h(e'))$.
Observe that all the aforementioned homomorphisms preserve neutral elements of groups. Thus, all of those homomorphisms are homomorphisms of appropriate retracts.
Therefore $(\gG(H,e),h)$ is a coequalizer of $\hH$-split parallel pair $(f,g)$ and $h$ is surjective. Thus, $\hH(\gG(H;h(e')))=H$ and $\hH$ preserves the coequalizers. Hence, the adjunction is monadic.
\end{proof}

To underline the meaning of the preceding theorem in the context of groups let us consider the following diagram
\begin{large}
\[
\xymatrix @C=30pt{
\grp \ar@<+0.5ex>[dr]^{\hH} \ar@<+0.5ex>[dd]^{\uU_{\grp}} & \\
 & \hrd\ , \ar@<+0.5ex>[dl]^{\uU_{\hrd}} \ar@<+0.5ex>[ul]^-{\mathrm{Gr}} \\
\Set \ar@<+0.5ex>[uu]^{\Gg} \ar@<+0.5ex>[ur]^{\Hh} & 
}
\]
\end{large}
where $U_{\grp}$ is a forgetful functor and $\Gg$ is its left adjoint, the free functor. 

The first observation is that all the opposite arrows are adjoints to each other.

The second observation is that a composition of functors,
$$
\xymatrix{\grp\ar[rr]^{\hH} &&\hrd\ar[rr]^{U_{\hrd}}&&\Set}
$$
is a forgetful functor $\uU_{\grp}$ since for any group $G$, $\hH(G)$ and $G$ are equal sets, and every homomorphism of groups $f$ is the same function as $\hH(f)$. These two observations leads to the following corollaries.
\begin{corollary}\label{cor3}
A functor $\mathrm{Gr}\circ \Hh:\Set\to \grp$ is a free functor, i.e. it is a left adjoint to the functor $U_{\grp}:\grp\to \Set$.
\end{corollary}
\begin{proof}
It is easy to show that a composition of two left adjoints is a left adjoint to the composition.
\end{proof}
\begin{corollary}\label{cor4}
For any set $X$, $(\mathrm{Gr}\circ\Hh)(X)\cong\Gg(X)$.
\end{corollary}
\begin{proof}
Since both functors $\mathrm{Gr}\circ\Hh$ and $\Gg$ are left adjoints to the forgetful functor, they are naturally isomorphic, see \cite[Corollary 1, page 85]{Mac:lane}.
\end{proof}
\section{Conclusions}\label{sect4}

To summarise, in the main part we have shown that a free functor from the category of sets to the category of groups is decomposable into two functors, through the category of heaps. The description of the free functor provides a method to construct a free group.

Sadly, in general, it is not an easy task to describe a coproduct of heaps. One intuitively knows it is a quotient of a free heap over the disjoint union. The choice of generators for the normal sub-heap is at least tricky, because one must at the same time deal with the allocations of a ternary operation in elements of the free heap, see the associativity rule  \eqref{rule:2}.

Fortunately, since we are interested in a composition of functors $\mathrm{Gr}\circ\Hh$, we only need to consider the coproduct of two heaps, the singleton heap and a free heap $\Hh(X)$, for any set $X$. Observe that a heap described on the singleton set is unique up to isomorphism. Thus, we can identify a singleton heap with a free heap $\Hh(\{*\})$. Now, by definition of $\mathrm{Gr}\circ\Hh$,
$(\mathrm{Gr}\circ\Hh)(X)=\gG(\Hh(X)\boxplus\Hh(\{*\});*)$, but $\Hh$ is a left adjoint functor to the forgetful functor, so it preserves coproducts. Therefore, we have that $\Hh(X)\boxplus\Hh(\{*\})\cong\Hh(X \sqcup \{*\})$. The construction of a free heap is well-known, see for example \cite{BrzRyb:mod}. Hence, we start with taking a set $X$, then consider a disjoint union with $\{*\}$, construct a free heap over $X \sqcup \{*\}$ and take a retract of that heap in $*$. The obtained retract is a free group. Even though one can argue that we still add a disjoint element $*$, in this setup, it has a proper algebraic interpretation in the category of heaps, i.e. taking a coproduct of a free heap with a singleton heap. Another difference is that one does not have to extend the set of generators by inverses, as inverses in the retract are words of the form $[*,w,*]$, for any $w\in \Hh(X \sqcup \{*\})$. 

Hopefully, the reader will find this observation as compelling as I do.

{\bf Acknowledgments.} The author is grateful to Tomasz Brzezi\'nski and Paolo Saracco for all the comments and advice.

\end{document}